\documentclass[12pt,leqno,amscd, amsfonts, amssymb,pstricks,verbatim]{amsart}
\usepackage{mathrsfs}
\usepackage{amsmath}
\usepackage{hyperref}
\usepackage{bookmark}
\usepackage{cite}

\def\la{\lambda}
\def\a{\alpha}
\def\r{\gamma}
\def\b{\beta}

\def\Z{\mathbb{Z}}
\def\N{\mathbb{N}}

\def\C{\mathbb{C}}

\numberwithin{equation}{section}
\newtheorem{theo}{Theorem}[section]
\newtheorem{defi}[theo]{Definition}

\newtheorem{lemm}[theo]{Lemma}
\newtheorem{prop}[theo]{Proposition}
\newtheorem{clai}{Claim}
\newtheorem{rema}[theo]{Remark}

\allowdisplaybreaks

\begin{document}

\title[Non-weight modules over the BMS-Kac-Moody algebra]{Non-weight modules over the BMS-Kac-Moody algebra}

\author{Qiufan Chen}

\address{Chen: Department of Mathematics, Shanghai Maritime University,
 Shanghai, 201306, China.}\email{chenqf@shmtu.edu.cn}

\author{Cong Guo}

\address{Guo: Department of Mathematics, Shanghai Maritime University,
 Shanghai, 201306, China.}\email{202331010011@stu.shmtu.edu.cn}



\thanks{This work is supported by National Natural Science Foundation of China (Grant Nos. 12271345 and 12361006). }

\begin{abstract} In this paper, we construct and classify a class of  non-weight modules over the BMS-Kac-Moody algebra,  which are free modules of rank one  when restricted to the universal enveloping algebra of the Cartan subalgebra (modulo center).  We give  the classification of such modules.  Moreover, the irreducibility and the isomorphism classes of these modules are determined.
\end{abstract}

\maketitle
\setcounter{tocdepth}{1}\tableofcontents
\begin{center}
\end{center}

\section{Introduction}

Throughout the paper, we denote by $\C ,\,\C^*,\,\Z,\,\Z^*,\,\N,\,\Z_+$ the sets of complex numbers, nonzero complex numbers, integers,  nonzero integers, nonnegative integers and positive integers, respectively. All vector spaces  are assumed to be  over $\C$. For a Lie algebra $\mathfrak{g}$, we use $U(\mathfrak{g})$ to denote the universal enveloping algebra of $\mathfrak{g}$. Let $\C[t]$ be the polynomial algebra in variable $t$ and  $\C[s,t]$ the polynomial algebra in  variables $s$ and $t$, respectively.

Infinite-dimensional symmetries play a prominent role in different areas of physics. It was discovered that the asymptotic symmetry group of three-dimensional Einstein gravity is generated by the BMS (Bondi-Metzner-Sachs) algebra, denoted by $\hat{\mathfrak B}$, which is an infinite-dimensional Lie algebra with a basis $\{L_m, M_m\mid m\in\Z\}$ and defining relations
\begin{equation*}
\aligned
&[L_m, L_n]=(n-m)L_{m+n}+\frac{C_L}{12}(m^3-m)\delta_{m+n,0},\\
&[L_m,M_n]=(n-m)M_{m+n}+\frac{C_M}{12}(m^3-m)\delta_{m+n,0},\\
&[M_m,M_n]=0,\quad \forall m, n\in\Z.
\endaligned
\end{equation*}
These $L_m$ and $M_m$ are called super-rotations and super-translations at quantum level, respectively. Observe that $\hat{\mathfrak B}$ is exactly  the $W(2,2)$ algebra, which was introduced in \cite{ZD} for the study of the classification of vertex operator algebras generated by weight 2 vectors. The  algebra $\hat{\mathfrak B}$ is crucial in establishing holography theory in asymptotic flat spacetimes, see\cite{B,BF,H}. Let $\mathfrak B$ denote the centerless  $\hat{\mathfrak B}$. Because of an extra anisotropic scaling symmetry, the BMS symmetry has expanded to a symmetry generated by a new type of BMS-Kac-Moody algebra with a nonvanishing $u(1)$ Kac-Moody current\cite{YC}. Motivated by \cite{YC}, an infinite-dimensional BMS-Kac-Moody algebra $\mathfrak L$ with two $u(1)$  Kac-Moody currents without central extensions was introduced in \cite{LS} and many irreducible restricted modules over $\mathfrak L$ were constructed therein. Explicitly, $\mathfrak L$ is a Lie algebra with a basis $\{L_m, M_m, S_m, I_m \mid m\in\Z\}$ and the nontrivial Lie brackets defined by
\begin{equation*}
\aligned
&[L_m,L_n]=(n-m)L_{m+n},\quad [L_m,M_n]=(n-m)M_{m+n},\\
&[L_m,S_n]=nS_{m+n},\quad [L_m,I_n]=nI_{m+n},\\
&[M_m,I_n]=-nS_{m+n},\quad \forall m,n\in\Z.
\endaligned
\end{equation*}
Besides the algebra $\mathfrak B$,  $\mathfrak L$  contains many other well-known subalgebras, such as the Lie subalgebra spanned by $\{L_m\mid m\in\Z\}$ is the Witt algebra and the Lie subalgebra spanned by $\{L_m, I_m\mid m\in\Z\}$ and   $\{L_m, S_m\mid m\in\Z\}$ are two copies of the (centerless) Heisenberg-Virasoro algebra. Moreover, the Lie subalgebra spanned by $\{I_m, S_m\mid m\in\Z\}$ is a commutative ideal of $\mathfrak L$. Notice that the Lie subalgebra $\C L_0\oplus\C M_0$  is  the Cartan subalgebra  (modulo center) of $\mathfrak{L}$. The construction of the Lie algebra $\mathfrak L$ suggests that such an infinite-dimensional algebra is important on its structure and representation theory which also has a connection with the 2d supersymmetric Galilean conformal algebra \cite{MR}, as it is the even part of the 2d supersymmetric Galilean conformal algebra.

In recent years, the study of non-weight modules has attracted more and more mathematicians' attention and $U(\mathfrak{h})$-free modules constitute an important class of non-weight modules, where $U(\mathfrak{h})$ is the universal
enveloping algebra of the Cartan subalgebra $\mathfrak{h}$. $U(\mathfrak{h})$-free modules (originally defined by an ``opposite condition" relative to weight modules) were constructed first for $\mathfrak{sl}_{n+1}$ by  J. Nilsson \cite{N} and independently by H. Tan and K. Zhao \cite{TZ2}. For a finite dimensional simple Lie algebra, J. Nilsson showed that  $U(\mathfrak{h})$-free modules can exist only when it is of type A or C \cite{N1}. In \cite{TZ}, H. Tan and K. Zhao classified  $U(\mathfrak{h})$-free module of rank one over the Witt algebra. Since then, many authors have constructed $U(\mathfrak{h})$-free modules for various Lie (super) algebras, see\cite{CTZ,CZ,CG,CGLW,CC,CH,CY,HCS,GMZ,YYX}. In the present paper,  we plan to study $U(\C L_0\oplus\C M_0)$-free modules of rank one over $\mathfrak{L}$. These lead to many new examples of irreducible modules over  $\mathfrak{L}$.

This paper is organized as follows. In Section 2, we construct a class of non-weight modules over $\mathfrak{L}$. Furthermore, we determine the isomorphism classes and irreducibility of these modules. In Section 3, we
show that these modules constitute a complete classification of $U(\C L_0\oplus\C M_0)$-free modules of rank one over $\mathfrak{L}$.

\section{Preliminaries}
Fix any $\lambda\in\C^*,\a\in\C$ and $h(t)\in\C[t]$. For any $m\in\Z$, we define
\begin{align*}
h_m(t)=m h(t)-m(m-1)\a\frac{h(t)-h(\a)}{t-\a}.
\end{align*}

For $\lambda\in\C^*,\a\in\C$ and $h(t)\in\C[t]$, denote by $\Phi_{\hat{\mathfrak B}}(\lambda,\a,h(t))=\C [s,t]$. In \cite{CG}, the $\hat{\mathfrak B}$-module structure on $\Phi_{\hat{\mathfrak B}}(\lambda,\a,h(t))$ is given by
\begin{equation*}
\aligned
&L_m\cdot f(s,t)=\lambda^m(s+h_m(t))\big)f(s-m,t)-m\lambda^m(t-m\a)\frac{\partial}{\partial t}\big(f(s-m,t)\big),\\
&M_m\cdot f(s,t)=\lambda^m(t-m\a)f(s-m,t),\\
&C_L\cdot f(s,t)=0,\quad C_M\cdot f(s,t)=0,\quad {\rm where}\  f(s,t)\in\C[s,t].
\endaligned
\end{equation*}

For later use, we need the following known results of $\Phi_{\hat{\mathfrak B}}(\lambda,\a,h(t))$.
\begin{prop}\label{pop1}(cf. \cite{CG,HCS})
Keep notations as above, then the following statements hold.
\begin{itemize}
\item[(1)] Any free $U(\C L_0\oplus\C M_0)$-module of rank one over $\hat{\mathfrak B}$  is isomorphic to $\Phi_{\hat{\mathfrak B}}(\lambda,\a,h(t))$ for $\lambda\in\C^*,\a\in\C, h(t)\in\C[t]$;
\item[(2)] $\Phi_{\hat{\mathfrak B}}(\lambda,\a,h(t))$ is irreducible  if and only if $\a\neq0$;
\item[(3)] $\Phi_{\hat{\mathfrak B}}(\lambda,\a,h(t))\cong \Phi_{\hat{\mathfrak B}}(\lambda^{\prime},\a^{\prime},g(t))\Longleftrightarrow (\lambda,\a,h(t))=(\lambda^{\prime},\a^{\prime},g(t))$.
\end{itemize}
\end{prop}
It is worthwhile to point out that though  $\mathfrak B$ is a quotient algebra of $\hat{\mathfrak B}$, they have the same module structure on $\C[s,t]$ and Proposition \ref{pop1} also holds for the corresponding $\mathfrak B$-module $\Phi_{\mathfrak B}(\lambda,\a,h(t))$, which can be extended to  the following $\mathfrak{L}$-module.
\begin{defi}\label{defi2.1}\rm For $\lambda\in\C^*,\a,\b,\r\in\C, h(t)\in\C[t]$ and $f(s,t)\in\C[s,t]$, define the action of $\mathfrak{L}$ on $\Phi(\lambda,\a,\b,\r,h(t)):=\C[s,t]$ as follows:
\begin{equation*}
\aligned
&L_m\cdot f(s,t)=\lambda^m(s+h_m(t))f(s-m,t)-m\lambda^m(t-m\a)\frac{\partial}{\partial t}\big(f(s-m,t)\big),\\
&M_m\cdot f(s,t)=\lambda^m(t-m\a)f(s-m,t),\\
&S_m\cdot f(s,t)=\la^m\b f(s-m,t),\\
&I_m\cdot f(s,t)=\la^m\big(\r-m\b \frac{h(t)-h(\a)}{t-\a}\big)f(s-m,t)+m\la^m\b \frac{\partial}{\partial t}\big(f(s-m,t)\big).
\endaligned
\end{equation*}
\end{defi}
%We remark that when $\b=\r=0$, the $\mathfrak{L}$-module $\Phi(\lambda,\a,0,0,h(t))$ degenerates to a $\mathfrak{B}$-module $\Phi_{\mathfrak B}(\lambda,\a,h(t))$. 
\begin{prop}
$\Phi(\lambda,\a,\b,\r,h(t))$ is an $\mathfrak{L}$-module under the action given in Definition \ref{defi2.1}. Moreover, $\Phi(\lambda,\a,\b,\r,h(t))$  is irreducible  if and only if $\a\neq0$ or $\b\neq0$. 
\end{prop}
\begin{proof}
For the first statement, in view of the  $\mathfrak B$-action, we know the following relations
\begin{equation*}
\aligned
&L_m\circ L_n-L_n\circ L_m=[L_m,L_n],\\
&L_m \circ M_n-M_n\circ L_m=[L_m,M_n],\\
&M_m\circ M_n-M_n\circ M_m=[M_m,M_n]
\endaligned
\end{equation*} hold on $\Phi(\lambda,\a,\b,\r,h(t))$.  For any $f(s,t)\in \C[s,t]$, explicit calculations show that
\begin{align*}
&L_m\cdot I_n\cdot f(s,t)-I_n\cdot L_m\cdot f(s,t)\\
&=L_m\cdot\big\{\la^n(\r-n\b\frac{h(t)-h(\a)}{t-\a})f(s-n,t)+n\la^n\b\frac{\partial}{\partial t}\big(f(s-n,t)\big)\big\}\\
&\quad-I_n\cdot\big\{\lambda^m(s+h_m(t))f(s-m,t)-m\lambda^m(t-m\a)\frac{\partial}{\partial t}\big(f(s-m,t)\big)\big\}\\
&=n\la^{m+n}\big(\gamma-(m+n)\beta\frac{h(t)-h(\a)}{t-\a}\big)f(s-m-n,t)+n(m+n)\la^{m+n}\beta\frac{\partial}{\partial t}\big(f(s-m-n,t)\big)\\
&=n\big\{\la^{m+n}\big(\gamma-(m+n)\beta \frac{h(t)-h(\a)}{t-\a}\big)f(s-m-n,t)+(m+n)\la^{m+n}\beta\frac{\partial}{\partial t}\big(f(s-m-n,t)\big)\big\}\\
&=[L_m,I_n]\cdot f(s,t),
\end{align*}
\begin{align*}
&L_m\cdot S_n\cdot f(s,t)-S_n\cdot L_m\cdot f(s,t)\\
&=L_m\cdot \big(\la^n\b f(s-n,t)\big)-S_n\cdot\big\{\lambda^m(s+h_m (t))f(s-m,t)-m\lambda^m(t-m\a)\frac{\partial}{\partial t}\big(f(s-m,t)\big)\big\}\\
&=n\la^{m+n}\b f(s-m-n,t)=[L_m,S_n]\cdot f(s,t)
\end{align*}
and 
\begin{align*}
&M_m\cdot I_n\cdot f(s,t)-I_n\cdot M_m\cdot f(s,t)\\
&=M_m\cdot\big\{\la^n(\gamma-n\beta \frac{h(t)-h(\a)}{t-\a}) f(s-n,t)+n\la^n\beta\frac{\partial}{\partial t}\big(f(s-n,t)\big)\big\}\\
&\quad- I_n\cdot\big(\la^m(t-m\a)f(s-m,t)\big)\\
&=\la^{m+n}\big\{(t-m\a)(\mu-n\beta \frac{h(t)-h(\a)}{t-\a})f(s-n-m,t)+n(t-m\a)\beta\frac{\partial}{\partial t}\big(f(s-n-m,t)\big)\big\}\\
&\quad-\la^{m+n}\big\{(t-m\a)(\gamma-n\beta \frac{h(t)-h(\a)}{t-\a})f(s-m-n,t)\\
&\quad+n\beta f(s-m-n,t)+n\beta(t-m\a)\frac{\partial}{\partial t}\big(f(s-m-n,t)\big)\big\}\\
&=-n\la^{m+n}\b f(s-m-n,t)=[M_m,I_n]\cdot f(s,t).
\end{align*}
The last five  relations
\begin{equation*}
\aligned
&I_m\circ S_n-S_n\circ I_m=[I_m,S_n],\\
&I_m\circ I_n-I_n\circ I_m=[I_m,I_n],\\
&M_m\circ M_n-M_n\circ M_m=[M_m,M_n],\\
&M_m\circ S_n-S_n\circ M_m=[M_m,S_n],\\
{\rm and}\quad & S_m\circ S_n-S_n\circ S_m=[S_m,S_n]
\endaligned
\end{equation*}can be checked easily, proving the first statement. 

For the second statement, assume first that  $\a\ne0$ and it follows directly from  \cite[Proposition 3.1]{CG} that $\Phi(\lambda,\a,\b,\r,h(t))$ is irreducible. Assume now that $\b\neq0$. Suppose that $W$ is a nonzero submodule of $\Phi(\lambda,\a,\b,\r,h(t))$. Let $f(s,t)$ be a nonzero polynomial in $W$ with the smallest $deg_{s}(f(s,t))$. By the action of $S_n$, we have $f(s-n,t)\in W$, which  indicates that $deg_{s}(f(s,t))=0,$ i.e., $f(t):=f(s,t)\in\C[t]$. From the actions of $L_0$ and $M_0$, we see that $f(t)\C[s,t]\supseteq W.$ This together with $I_1\cdot f(t)\in W$ forces $f^\prime(t)\in W$, which immediately gives that $W=\C[s,t].$ Hence  $\Phi(\lambda,\a,\b,\r,h(t))$ is irreducible. Finally, consider the case $\a=\b=0$. It is apparent that for any $i\in\N$, $t^i\C[s,t]$ is a submodule of $\Phi(\lambda,0,0,\r,h(t))$ for all $\la\in\C^*,\r\in\C$ and $h(t)\in\C[t]$. And $\Phi(\lambda,0,0,\r,h(t))$ has the following submodule filtration $\cdots \subseteq t^{i+1}\C[s,t]\subseteq t^{i}\C[s,t] \subseteq\cdots \subseteq t\C[s,t]\subseteq \C[s,t]$. For any $0\neq t^ig(s)\in t^{i}\C[s,t]/t^{i+1}\C[s,t]$ and $m\in\Z$, we have
\begin{equation*}
\aligned
&L_m\cdot t^ig(s)\equiv \lambda^m(s+m(h(0)-i))t^ig(s-m)\quad({\rm mod}\,\,t^{i+1}\C[s,t]),\\
&I_m\cdot t^ig(s)=\la^{m}\r t^ig(s-m),\\
&M_m\cdot t^ig(s)=S_m\cdot t^ig(s)=0.
\endaligned
\end{equation*}
Therefore, the quotient module $t^{i}\C[s,t]/t^{i+1}\C[s,t]$ can be viewed as a module over the (centerless) Heisenberg-Virasoro algebra. It follows from \cite{LZ} that 
$t^{i}\C[s,t]/t^{i+1}\C[s,t]$ is irreducible if and only if $h(0)\neq i$ or $\r\neq0$. We complete the proof.
\end{proof}
The following proposition gives a characterization of  isomorphic relations between the  $\mathfrak{L}$-modules constructed in Definition \ref{defi2.1}.
\begin{prop}
Let $\lambda, \lambda^{\prime}\in\mathbb{C}^*,\a,\a^{\prime},\b, \b^{\prime},\r,\r^\prime\in\mathbb{C},h(t),q(t)\in\C[t]$. Then
\begin{align*}\label{2.3}
\Phi(\lambda,\a,\b,\r,h(t)) &\cong \Phi(\lambda^{\prime},\a^{\prime},\b^{\prime},\r^\prime,q(t))\Longleftrightarrow (\lambda,\a,\b,\r,h(t))=(\lambda^{\prime},\a^{\prime},\b^{\prime},\r^\prime,q(t)).
\end{align*}
\end{prop}
\begin{proof} 
The sufficiency is obvious and it suffices to show the necessity. Let $\varphi:\Phi(\lambda,\a,\mu, h(t))\to\Phi(\lambda^{\prime},\a^{\prime},\mu^{\prime},q(t))$ be an isomorphism of $\mathfrak{L}$-modules. Viewing $\Phi(\lambda,\a,\b, h(t))$ and $\Phi(\lambda^{\prime},\a^{\prime},\b^{\prime},q(t))$ as $\mathfrak{B}$-modules, we get $(\lambda,\a,h(t))=(\lambda^{\prime},\a^{\prime},q(t))$ by Proposition \ref{pop1} (3). For any  $f(s,t)\in\C[s,t]$,  we have
\begin{equation*}
\varphi(f(s,t))=\varphi(f(L_0,M_{0})\cdot1)=f(L_0,M_0)\cdot\varphi(1)=f(s,t)\varphi(1). \end{equation*}
In particular,  we get $1=\varphi\big(\varphi^{-1}(1)\big)=\varphi^{-1}(1)\varphi(1),$ which means that $\varphi(1)$ is  invertible  in $\C[s,t]$. Since nonzero constants are the only invertible elements in $\C[s,t],$  it follows that $\varphi(1)\in\C^*$. From
\begin{equation*}\label{add2}
\aligned
&\b\varphi(1)=\varphi(\b)=\varphi(S_0\cdot1)=S_0\cdot \varphi(1)=\varphi(1)\b^{\prime},\\
&\r\varphi(1)=\varphi(\r)=\varphi(I_0\cdot1)=I_0\cdot \varphi(1)=\varphi(1)\r^{\prime},
\endaligned
\end{equation*}
we infer that $(\b,\r)=(\b^{\prime},\r^{\prime})$. This completes the proof.
\end{proof}

\section{Main result}
The main result of the present paper is to classify all modules over $\mathfrak{L}$ whose restrictions to $U(\C L_0 \oplus\C M_0)$ are
free of rank one. Before presenting the main result, we first give a lemma, which  can be easily shown  by induction on $i$.
\begin{lemm}\label{lemma1}
The following formulae  hold in $U(\mathfrak{L}),$ for any $m\in\Z$ and $i\in\N,$
\begin{align*}
I_m L_0^i&=(L_0-m)^iI_m,\quad S_m L_0^i=(L_0-m)^iS_m,\quad I_m M_0^i=M_0^i I_m+im M_0^{i-1}S_m.
\end{align*}
\end{lemm}
\begin{theo}\label{theo1} Let $M$ be an $\mathfrak{L}$-module such that the restriction of $U(\mathfrak{L})$ to $U(\C L_0 \oplus\C M_0)$ is  free of rank one. Then $M\cong \Phi(\lambda,\a,\b,\r,h(t))$ for some $\lambda\in\C^*,\ \a,\b,\r\in\C,\ h(t)\in\C[t]$.
\end{theo}

\begin{proof} Assume that $M=U(\C L_0 \oplus\C M_0)$. Viewed as a $\mathfrak B$-module, following from  Proposition \ref{pop1} (1), we have
\begin{align}
\label{111}L_m\cdot f(L_0,M_0)&=\lambda^m\big(L_0+h_m (M_0)\big)f(L_0-m,M_0)\nonumber\\
&\ \ \ \ -m\la^m (M_0-m\a)\frac{\partial}{\partial M_0}\big(f(L_0-m,M_0)\big),\\
\label{222}M_m\cdot f(L_0,M_0)&=\lambda^m (M_0-m\a) f(L_0-m,M_0),
\end{align}
where $f(L_0,M_0)\in U(\C L_0 \oplus\C M_0),\lambda\in\C^*,\a\in\C,m\in\Z$ and $h(t)\in\C[t]$.

For any $m\in\Z$, denote $a_m(L_0,M_0):=S_m\cdot1$ and $b_m(L_0,M_0):=I_m\cdot1$, respectively.  Take any
$$f(L_0, M_0)=\sum_{j,k\in\N}c_{j,k}L_{0}^j M_{0}^k\in U(\C L_0\oplus \C M_0),\quad {\rm where}\ c_{j,k}\in\C.$$ By  Lemma \ref{lemma1}, we have
\begin{align*}
S_m\cdot f(L_0, M_0)&=S_m\cdot \sum_{j,k\in\N}c_{j,k}L_{0}^j M_{0}^k=\sum_{j,k\in\N}c_{j,k}(L_{0}-m)^{j} S_m\cdot M_{0}^{k}\nonumber\\
&=\sum_{j,k\in\N}c_{j,k}(L_{0}-m)^{j} M_0^k a_m(L_0,M_0)=f(L_0-m, M_0) a_m(L_0,M_0)
\end{align*}
and
\begin{align*}
I_m\cdot f(L_0, M_0)&=I_m\cdot \sum_{j,k\in\N}c_{j,k}L_{0}^j M_{0}^k=\sum_{j,k\in\N}c_{j,k}(L_{0}-m)^{j} I_m\cdot M_{0}^{k}\nonumber\\
&=\sum_{j,k\in\N}c_{j,k}(L_{0}-m)^{j} M_{0}^{k}b_m(L_0,M_0)+\sum_{j,k\in\N}c_{j,k}k m (L_{0}-m)^{j}M_{0}^{k-1}a_m(L_0,M_0)\nonumber\\
&=f(L_0-m, M_0)b_m(L_0,M_0)+m\frac{\partial}{\partial M_0}\big(f(L_0-m, M_0)\big)a_m(L_0,M_0).
\end{align*}
Thus, the actions of $S_m$ and $I_m$ on $M$ are completely determined by $a_m(L_0,M_0)$ and $b_m(L_0,M_0)$, respectively. For this reason, in what follows we only need to determine $a_m(L_0,M_0)$ and $b_m(L_0,M_0)$ for all $m\in\Z$.
\begin{clai}$\b:=a_0(L_0,M_0)\in\C,\r:=b_0(L_0,M_0)\in\C.$\end{clai} 
Equation $[M_n,S_0]\cdot1=0$ along with \eqref{222}  gives
\begin{align*}
\la^{n} (M_0-n\a)a_{0}(L_0-n,M_0)=\la^n (M_0-n\a)a_{0}(L_0,M_0),
\end{align*}
which yields that $a_0(M_0):=a_0(L_0,M_0)\in\C[M_0]$. Combining this with $[L_1,S_0]\cdot1=0$ and  \eqref{111}, we obtain $\la(M_0-\a)a_{0}^{\prime}(M_0)=0$. This shows that $\b:=a_0(M_0)\in\C$. Similarly, using \eqref{111}, \eqref{222}, $[M_n,I_0]\cdot1=0$ and $[L_1,I_0]\cdot1=0$, one can easily get $\r:=b_0(L_0,M_0)\in\C$.
\begin{clai}$a_m(L_0,M_0)=\la^m \b$ {\rm for all} $m\in\Z.$\end{clai}
By applying $[M_n,S_m]=0$ to $1$, one has
\begin{align*}
\la^n(M_0-n\a)a_m(L_0-n,M_0)&=\la^n(M_0-n\a)a_m(L_0,M_0),
\end{align*}
which forces  $a_m(M_0):=a_m(L_0,M_0)\in\C[M_0]$ for all $m\in\Z$. This entails us to assume that $a_m(M_0)=\sum_{i=0}^{k_m}c_{m,i}M_0^i$ with $c_{m,i}\in\C$ and $m\in\Z$. Now we compute
\begin{align*}
[L_{-m},S_m]\cdot1&=\la^{-m}ma_m(M_0)+\la^{-m}m(M_0+m\a)a_m^\prime(M_0)\\
&=\la^{-m}m\sum_{i=0}^{k_m}c_{m,i}M_0^i+\la^{-m}m(M_0+m\a)\sum_{i=0}^{k_m}ic_{m,i}M_0^{i-1}\\
&=\la^{-m}m(k_m+1)c_{m,k_m}M_0^{k_m}\quad({\rm mod}\,\,\oplus_{i=0}^{k_{m}-1}\C M_{0}^{i})\\
&=mS_0\cdot1=m\b.
\end{align*}
Hence it follows that $k_m=0$ and  $a_m(M_0)=c_{m,0}=\la^m\b$ for all $m\in\Z^*$. Together with $a_0(M_0)=\b$, we see the claim holds.
\begin{clai} $b_m(L_0,M_0)=\la^m\big(\r-m\b \frac{h(t)-h(\a)}{t-\a}\big)\  {\rm for\ all}\ m\in\Z.$ \end{clai}
Owing to Claim $2$ and $[M_n,I_m]\cdot1=-mS_{m+n}\cdot1$, we obtain  $b_m(L_0-n,M_0)=b_m(L_0,M_0)$ for all $m\in\Z$, forcing $b_m(M_0):=b_m(L_0,M_0)\in\C[M_0]$ for all $m\in\Z$. Explicit calculations show that
\begin{align}\label{3.9}
[L_n,I_m]\cdot1&=L_n \cdot b_m(M_0)-I_m\cdot \la^n\big(L_0+h_n(M_0)\big)\nonumber\\
&=m\la^n b_m(M_0)-n\la^n(M_0-n\a)b_m^\prime(M_0)-m\la^{m+n}\b h_n^\prime(M_0)\nonumber\\
&=mI_{m+n}\cdot1=m b_{m+n}(M_0)
\end{align}
and 
\begin{align}\label{3.10}
[I_n,I_m]\cdot1&=I_n \cdot b_m(M_0)-I_m\cdot b_n(M_0)=n\la^n\b b_m^\prime(M_0)-m\la^m\b b_n^\prime(M_0)=0.
\end{align}
In the case $\b=0$, \eqref{3.9} simply becomes 
$$m\la^n b_m(M_0)-n\la^n(M_0-n\a)b_m^\prime(M_0)=m b_{m+n}(M_0).$$
Writing $b_m(M_0)=\sum_{i=0}^{l_m}d_{m,i}M_0^i$ with $d_{m,i}\in\C$ and inserting this into the above equality, we obtain that $l_m=0$ and $b_m(M_0)=\la^m\r$ for all $m\in\Z$. Next consider the case $\b\neq0$. By \eqref{3.10}, we have $n\la^n b_m^\prime(M_0)=m\la^m b_n^\prime(M_0)$. Letting  $n=1$, we have
\begin{align}\label{3.11}
b_m(M_0)=m\la^{m-1}b_1(M_0)+A_m
\end{align}
with $A_m\in\C, A_1=0$ and $A_0=\r$. 
%Letting $m=-1$ in the first equality of \eqref{3.11}, we have $b_{-1}(M_0)=-\la^{-2}b_1(M_0)+A_{-1}$.
%Substituting this into the second equality of \eqref{3.11} and  equating the two expressions of $b_m(M_0)$, we obtain
%\begin{align}\label{3.13}
%A_m=B_m-m\la^{m+1}A_{-1},\ \ \ \ \forall m\in\Z.
%\end{align}
By taking  $(m,n)=(-1,1)$ and  $(m,n)=(1,-1)$ in \eqref{3.9}, we respectively obtain
\begin{align}\label{2.3}
\b h_{1}^{\prime}(M_0)=-\lambda^{-1}b_{1}(M_0)-\lambda^{-1}(M_0-\a)b_{1}^{\prime}(M_0)+\lambda A_{-1}-\r
\end{align}
and
\begin{align*}
\b h_{-1}^{\prime}(M_0)=\lambda^{-1}b_{1}(M_0)+\lambda^{-1}(M_0+\a)b_{1}^{\prime}(M_0)-\r.
\end{align*}
It follows from these, \eqref{3.11} and  by choosing  $n=1$ and $n=-1$ in \eqref{3.9} that
\begin{align*}
&A_{m+1}=\la A_m+\la^{m+1}\r-\la^{m+2} A_{-1},\nonumber\\
&A_{m-1}=\la^{-1} A_m+\la^{m-1}\r,\ \ \ \ \forall m\in\Z^*, 
\end{align*}
from which one can inductively infer that 
\begin{align}\label{3.12}
A_m=(1-m)\lambda^{m}\r,\ \ \ \ \forall m\in\Z. 
\end{align} In particular, $A_{-1}=2\la^{-1}\r$. Now \eqref{2.3} becomes
$$\b h^{\prime}(M_0)=-\lambda^{-1}b_{1}(M_0)-\lambda^{-1}(M_0-\a)b_{1}^{\prime}(M_0)+\r,$$
which means $\b h(M_0)=-\la^{-1}b_1(M_0)(M_0-\a)+\r M_0+\theta$ for some $\theta\in\C$. Clearly, $\theta=\b h(\a)-\r\a$. Thus, we have
\begin{align*}
b_1(M_0)=\frac{\la\big(\r M_0+\b h(\a)-\r\a-\b h(M_0)\big)}{M_0-\a}.
\end{align*}
Combining this with \eqref{3.11} and \eqref{3.12}, we conclude that 
$$b_m(M_0)=\la^m\big(\r-m\b \frac{h(M_0)-h(\a)}{M_0-\a}\big),\ \ \ \ \forall m\in\Z.$$ This completes the proof of Claim 3 and thereby Theorem \ref{theo1}.
\end{proof}
\begin{rema}
As for any non-zero element  $v\in\Phi(\lambda,\a,\b,\r,h(t))$, there exists a positive integer $m$ such that $L_m$ acts non-trivially on $v$  and hence $\Phi(\lambda,\a,\b,\r,h(t))$ is not isomorphic to a  restricted module in \cite{LS}. Then $\Phi(\lambda,\a,\b,\r,h(t))$ is a new non-weight $\mathfrak L$-module.

\end{rema}

%\subsection*{Acknowledgements}
%The authors would like to thank the referee for his/her helpful comments and suggestion.

\end{document}